\documentclass[numbers=enddot,12pt,final,onecolumn,notitlepage]{scrartcl}%
\usepackage[headsepline,footsepline,manualmark]{scrlayer-scrpage}
\usepackage{amssymb}
\usepackage{amsmath}
\usepackage{amsthm}
\usepackage{framed}
\usepackage{comment}
\usepackage{color}
\usepackage[breaklinks=True]{hyperref}
\usepackage{hyperxmp}
\usepackage[sc]{mathpazo}
\usepackage[T1]{fontenc}
\usepackage{needspace}
\usepackage{tabls}
\usepackage[type={CC}, modifier={zero}, version={1.0},]{doclicense}
\providecommand{\U}[1]{\protect\rule{.1in}{.1in}}
\theoremstyle{definition}
\newtheorem{theo}{Theorem}[section]
\newenvironment{theorem}[1][]
{\begin{theo}[#1]\begin{leftbar}}
{\end{leftbar}\end{theo}}
\newtheorem{lem}[theo]{Lemma}

\newtheorem{prop}[theo]{Proposition}

\newtheorem{defi}[theo]{Definition}

\newtheorem{remk}[theo]{Remark}

\newtheorem{coro}[theo]{Corollary}

\newtheorem{conv}[theo]{Convention}

\newtheorem{quest}[theo]{Question}
\newenvironment{question}[1][]
{\begin{quest}[#1]\begin{leftbar}}
{\end{leftbar}\end{quest}}
\newtheorem{warn}[theo]{Warning}
\newenvironment{warning}[1][]
{\begin{warn}[#1]\begin{leftbar}}
{\end{leftbar}\end{warn}}
\newtheorem{conj}[theo]{Conjecture}

\newtheorem{exam}[theo]{Example}
\newenvironment{example}[1][]
{\begin{exam}[#1]\begin{leftbar}}
{\end{leftbar}\end{exam}}
\newenvironment{statement}{\begin{quote}}{\end{quote}}

\let\sumnonlimits\sum
\let\prodnonlimits\prod
\let\cupnonlimits\bigcup
\let\capnonlimits\bigcap
\renewcommand{\sum}{\sumnonlimits\limits}
\renewcommand{\prod}{\prodnonlimits\limits}
\renewcommand{\bigcup}{\cupnonlimits\limits}
\renewcommand{\bigcap}{\capnonlimits\limits}
\excludecomment{verlong}
\includecomment{vershort}
\excludecomment{noncompile}

\ihead{Principal minors}
\ohead{page \thepage}
\begin{document}

\title{On the principal minors of the powers of a matrix}
\author{Darij Grinberg\thanks{Drexel University, Korman Center, 15 S 33rd Street,
Philadelphia PA, 19104, USA}}
\date{30 June 2026}
\maketitle

\begin{abstract}
\textbf{Abstract.} We show that if $A$ is an $n\times n$-matrix, then the
diagonal entries of each power $A^{m}$ are uniquely determined by the
principal minors of $A$, and can be written as universal (integral)
polynomials in the latter. Furthermore, if the latter all equal $1$, then so
do the former. These results are inspired by Problem B5 on the Putnam contest
2021, and shed a new light on the behavior of minors under matrix multiplication.


\textbf{MSC2020 classes:} 11C20, 14M12, 15A15.

\end{abstract}

\section{Introduction}

Let $R$ be a commutative ring. Let $A$ be an $n\times n$-matrix over $R$,
where $n$ is a nonnegative integer.

A \emph{principal submatrix} of $A$ means a matrix obtained from $A$ by
removing some rows and the corresponding columns (i.e., removing the $i_{1}%
$-th, $i_{2}$-th, $\ldots$, $i_{k}$-th rows and the $i_{1}$-th, $i_{2}$-th,
$\ldots$, $i_{k}$-th columns for some choice of $k$ integers $i_{1}%
,i_{2},\ldots,i_{k}$ satisfying $1\leq i_{1}<i_{2}<\cdots<i_{k}\leq n$). In
particular, $A$ itself is a principal submatrix of $A$ (obtained for $k=0$).

A \emph{principal minor} of $A$ means the determinant of a principal submatrix
of $A$. In particular, each diagonal entry of $A$ is a principal minor of $A$
(being the determinant of a principal submatrix of size $1\times1$). In total,
$A$ has $2^{n}$ principal minors, including its own determinant $\det A$ as
well as the trivial principal minor $1$ (obtained as the determinant of a
$0\times0$ matrix, which is what remains when all rows and columns are removed).

Problem B5 on the \href{https://kskedlaya.org/putnam-archive/2021.pdf}{Putnam
contest 2021} asked for a proof of the following:

\begin{theorem}
\label{thm.putnam}Assume that $R=\mathbb{Z}$. Assume that each principal minor
of $A$ is odd. Then, each principal minor of $A^{m}$ is odd whenever $m$ is a
nonnegative integer.
\end{theorem}

Without giving the solution away, it shall be noticed that essentially only
one proof is known (see \cite{Kedlaya} or \cite{Putnam-official-2021} for it),
and it is not as algebraic as the statement of Theorem \ref{thm.putnam} might
suggest. In particular, it is unclear if the theorem remains valid if
\textquotedblleft odd\textquotedblright\ is replaced by \textquotedblleft
congruent to $1$ modulo $4$\textquotedblright, or if $R$ is replaced by
another ring; the official solution (most of which originates in a result by
Dobrinskaya \cite[Lemma 3.3]{MasPan}) certainly does not apply to such
extensions. An approach that is definitely doomed is to try expressing the
principal minors of a power $A^{m}$ in terms of those of $A$. The following
example shows that the latter do not uniquely determine the former:

\begin{example}
\label{exa.CD}Set%
\[
C:=%
\begin{pmatrix}
a & b & 1 & 1\\
c & d & 1 & 1\\
1 & 1 & p & q\\
1 & 1 & r & s
\end{pmatrix}
\ \ \ \ \ \ \ \ \ \ \text{and}\ \ \ \ \ \ \ \ \ \ D:=%
\begin{pmatrix}
a & b & 1 & 1\\
c & d & 1 & 1\\
1 & 1 & p & r\\
1 & 1 & q & s
\end{pmatrix}
\]
for some $a,b,c,d,p,q,r,s\in R$. Then, the matrices $C$ and $D$ have the same
principal minors, but their squares $C^{2}$ and $D^{2}$ differ in their
$\left\{  2,3\right\}  $-principal minor (i.e., their principal minor obtained
by removing the $1$-st and $4$-th rows and columns) unless $\left(
q-r\right)  \left(  b-c\right)  =0$. Thus, the principal minors of the square
of a matrix are not uniquely determined by the principal minors of the matrix itself.
\end{example}

This example is inspired by \cite[Example 3]{HarLoe}, where further related
discussion of matrices with equal principal minors can be found.

\section{Nevertheless...}

However, not all is lost. Among the principal minors of $A^{m}$, the simplest
ones (besides $1$) are those of size $1\times1$, that is, the diagonal entries
of $A^{m}$. It turns out that these diagonal entries are indeed uniquely
determined by the principal minors of $A$, and even better, they can be
written as universal polynomials\footnote{A \textquotedblleft universal
polynomial\textquotedblright\ means a polynomial with integer coefficients
that depends neither on $A$ nor on $R$ (but can depend on $m$ as well as on
the location of the diagonal entry).} in the latter. That is, we have the
following:\footnote{An \emph{integer polynomial} means a polynomial with
integer coefficients.}

\begin{theorem}
\label{thm.main1}Let $n$ and $m$ be nonnegative integers, and let
$i\in\left\{  1,2,\ldots,n\right\}  $. Then, there exists an integer
polynomial $P_{n,i,m}$ in $2^{n}$ indeterminates that is independent of $R$
and $A$, and that has the following property: If $A$ is any $n\times n$-matrix
over any commutative ring $R$, then the $i$-th diagonal entry of $A^{m}$ can
be obtained by substituting the principal minors of $A$ into $P_{n,i,m}$. In
particular, the principal minors of $A$ uniquely determine this entry.
\end{theorem}

Let us verify this for $m=2$: If we denote the $\left(  i,j\right)  $-th entry
of a matrix $B$ by $B_{i,j}$, then each diagonal entry of $A^{2}$ has the form%
\begin{align*}
\left(  A^{2}\right)  _{i,i}  &  =\sum_{j=1}^{n}A_{i,j}A_{j,i}=A_{i,i}%
^{2}+\sum_{j\neq i}\underbrace{A_{i,j}A_{j,i}}_{=A_{i,i}A_{j,j}-\det\left(
\begin{array}
[c]{cc}%
A_{i,i} & A_{i,j}\\
A_{j,i} & A_{j,j}%
\end{array}
\right)  }\\
&  =A_{i,i}^{2}+\sum_{j\neq i}\left(  A_{i,i}A_{j,j}-\det\left(
\begin{array}
[c]{cc}%
A_{i,i} & A_{i,j}\\
A_{j,i} & A_{j,j}%
\end{array}
\right)  \right)  ,
\end{align*}
which is visibly an integer polynomial in the principal minors of $A$ (since
all the $A_{i,i}$ and $A_{j,j}$ and $\det\left(
\begin{array}
[c]{cc}%
A_{i,i} & A_{i,j}\\
A_{j,i} & A_{j,j}%
\end{array}
\right)  $ are principal minors of $A$). This verifies Theorem \ref{thm.main1}
for $m=2$. Such explicit computations remain technically possible for higher
values of $m$, but become longer and more cumbersome as $m$ increases.

The goal of this note is to prove Theorem \ref{thm.main1}. We will first show
the following theorem, which looks weaker but is essentially equivalent:

\begin{theorem}
\label{thm.main2}Let $n$ and $m$ be nonnegative integers. Let $R$ be a
commutative ring. Let $A$ be an $n\times n$-matrix over $R$. Let $\mathcal{P}$
be the subring of $R$ generated by all principal minors of $A$. Then, all
diagonal entries of $A^{m}$ belong to $\mathcal{P}$.
\end{theorem}

Before we prove this, let us explain how Theorem \ref{thm.main1} can be easily
derived from Theorem \ref{thm.main2}:

\begin{proof}
[Proof of Theorem \ref{thm.main1} using Theorem \ref{thm.main2}.]The notation
$B_{i,j}$ shall denote the $\left(  i,j\right)  $-th entry of any matrix $B$.

Let $\mathbf{R}$ be the polynomial ring $\mathbb{Z}\left[  x_{i,j}%
\ \mid\ 1\leq i\leq n\text{ and }1\leq j\leq n\right]  $ in $n^{2}$
independent indeterminates $x_{i,j}$ over $\mathbb{Z}$. (For instance, if
$n=2$, then $\mathbf{R}=\mathbb{Z}\left[  x_{1,1},\ x_{1,2},\ x_{2,1}%
,\ x_{2,2}\right]  $.) Let $\mathbf{A}$ be the $n\times n$-matrix over
$\mathbf{R}$ whose $\left(  i,j\right)  $-th entry is $x_{i,j}$ for each
$\left(  i,j\right)  \in\left\{  1,2,\ldots,n\right\}  ^{2}$. This matrix
$\mathbf{A}$ is known as \textquotedblleft the general $n\times n$%
-matrix\textquotedblright, since any matrix $A$ over any commutative ring can
be obtained from it by substituting appropriate elements (viz., the entries of
$A$) for the variables $x_{i,j}$. This very property will be crucial to the
argument that follows.

Let $\mathbf{p}_{1},\mathbf{p}_{2},\ldots,\mathbf{p}_{2^{n}}$ be the $2^{n}$
principal minors of $\mathbf{A}$ (numbered in some order). Let $\mathcal{P}$
denote the subring of $\mathbf{R}$ generated by all these principal minors of
$\mathbf{A}$. Theorem \ref{thm.main2} (applied to $\mathbf{R}$ and
$\mathbf{A}$ instead of $R$ and $A$) shows that all diagonal entries of
$\mathbf{A}^{m}$ belong to $\mathcal{P}$. In other words, for each
$i\in\left\{  1,2,\ldots,n\right\}  $, we have $\left(  \mathbf{A}^{m}\right)
_{i,i}\in\mathcal{P}$.

Fix $i\in\left\{  1,2,\ldots,n\right\}  $. As we just showed, we have $\left(
\mathbf{A}^{m}\right)  _{i,i}\in\mathcal{P}$. In other words, there exists an
integer polynomial $P_{n,i,m}$ in $2^{n}$ indeterminates such that $\left(
\mathbf{A}^{m}\right)  _{i,i}=P_{n,i,m}\left(  \mathbf{p}_{1},\mathbf{p}%
_{2},\ldots,\mathbf{p}_{2^{n}}\right)  $ (since $\mathcal{P}$ is the subring
of $\mathbf{R}$ generated by $\mathbf{p}_{1},\mathbf{p}_{2},\ldots
,\mathbf{p}_{2^{n}}$). Consider this polynomial $P_{n,i,m}$; note that it is
independent of $R$ and $A$ (by its very construction).

Now, consider a commutative ring $R$ and an $n\times n$-matrix $A$ over $R$.
Let $p_{1},p_{2},\ldots,p_{2^{n}}$ be the $2^{n}$ principal minors of $A$
(numbered in the same order as $\mathbf{p}_{1},\mathbf{p}_{2},\ldots
,\mathbf{p}_{2^{n}}$). Let $f:\mathbf{R}\rightarrow R$ be the $\mathbb{Z}%
$-algebra homomorphism that sends each indeterminate $x_{i,j}$ to the $\left(
i,j\right)  $-th entry $A_{i,j}$ of $A$. This homomorphism $f$ therefore sends
each entry of the matrix $\mathbf{A}$ to the corresponding entry of $A$, and
thus also sends each principal minor of $\mathbf{A}$ to the corresponding
principal minor of $A$ (since a principal minor is a certain signed sum of
products of entries of the matrix). In other words,%
\begin{equation}
f\left(  \mathbf{p}_{j}\right)  =p_{j}\ \ \ \ \ \ \ \ \ \ \text{for each }%
j\in\left\{  1,2,\ldots,2^{n}\right\}  . \label{pf.thm.main1.1}%
\end{equation}

However, $P_{n,i,m}$ is an integer polynomial, and thus \textquotedblleft
commutes\textquotedblright\ with any $\mathbb{Z}$-algebra homomorphism --
i.e., if $a_{1},a_{2},\ldots,a_{2^{n}}$ are any $2^{n}$ elements of a
commutative ring, and if $g$ is any $\mathbb{Z}$-algebra homomorphism out of
that ring, then%
\[
g\left(  P_{n,i,m}\left(  a_{1},a_{2},\ldots,a_{2^{n}}\right)  \right)
=P_{n,i,m}\left(  g\left(  a_{1}\right)  ,\ g\left(  a_{2}\right)
,\ \ldots,\ g\left(  a_{2^{n}}\right)  \right)  .
\]
Applying this to $a_{i}=\mathbf{p}_{i}$ and $g=f$, we obtain%
\begin{align*}
f\left(  P_{n,i,m}\left(  \mathbf{p}_{1},\mathbf{p}_{2},\ldots,\mathbf{p}%
_{2^{n}}\right)  \right)   &  =P_{n,i,m}\left(  f\left(  \mathbf{p}%
_{1}\right)  ,\ f\left(  \mathbf{p}_{2}\right)  ,\ \ldots,\ f\left(
\mathbf{p}_{2^{n}}\right)  \right) \\
&  =P_{n,i,m}\left(  p_{1},p_{2},\ldots,p_{2^{n}}\right)
\ \ \ \ \ \ \ \ \ \ \left(  \text{by (\ref{pf.thm.main1.1})}\right)  .
\end{align*}
In view of $\left(  \mathbf{A}^{m}\right)  _{i,i}=P_{n,i,m}\left(
\mathbf{p}_{1},\mathbf{p}_{2},\ldots,\mathbf{p}_{2^{n}}\right)  $, we can
rewrite this as%
\begin{equation}
f\left(  \left(  \mathbf{A}^{m}\right)  _{i,i}\right)  =P_{n,i,m}\left(
p_{1},p_{2},\ldots,p_{2^{n}}\right)  . \label{pf.thm.main1.2}%
\end{equation}

However, the $\mathbb{Z}$-algebra homomorphism $f$ sends each entry of the
matrix $\mathbf{A}$ to the corresponding entry of $A$, and therefore also
sends each entry of the matrix $\mathbf{A}^{m}$ to the corresponding entry of
$A^{m}$ (since the entries of $\mathbf{A}^{m}$ are certain sums of products of
entries of $\mathbf{A}$, whereas the entries of $A^{m}$ are the same sums of
products of entries of $A$). In other words, $f\left(  \left(  \mathbf{A}%
^{m}\right)  _{u,v}\right)  =\left(  A^{m}\right)  _{u,v}$ for any
$u,v\in\left\{  1,2,\ldots,n\right\}  $. Thus, in particular, $f\left(
\left(  \mathbf{A}^{m}\right)  _{i,i}\right)  =\left(  A^{m}\right)  _{i,i}$.
Comparing this with (\ref{pf.thm.main1.2}), we obtain $\left(  A^{m}\right)
_{i,i}=P_{n,i,m}\left(  p_{1},p_{2},\ldots,p_{2^{n}}\right)  $. In other
words, the $i$-th diagonal entry of $A^{m}$ can be obtained by substituting
the principal minors of $A$ into $P_{n,i,m}$ (since $p_{1},p_{2}%
,\ldots,p_{2^{n}}$ are these principal minors of $A$). This proves Theorem
\ref{thm.main1}.
\end{proof}

\section{Notations}

In order to prove Theorem \ref{thm.main2}, we will need some more notations
regarding matrices and their minors:

\begin{itemize}
\item If $m\in\mathbb{Z}$, then $\left[  m\right]  $ shall denote the set
$\left\{  1,2,\ldots,m\right\}  $. (This set is empty if $m\leq0$.)

\item If $B$ is a $u\times v$-matrix and if $i\in\left[  u\right]  $ and
$j\in\left[  v\right]  $, then $B_{i,j}$ shall denote the $\left(  i,j\right)
$-th entry of $B$.

\item If $u$ and $v$ are two nonnegative integers, and if $a_{i,j}$ is an
element of a ring for each $i\in\left[  u\right]  $ and $j\in\left[  v\right]
$, then the notation $\left(  a_{i,j}\right)  _{1\leq i\leq u,\ 1\leq j\leq
v}$ means the $u\times v$-matrix whose $\left(  i,j\right)  $-th entry is
$a_{i,j}$ for all $i\in\left[  u\right]  $ and $j\in\left[  v\right]  $.

\item If $B$ is a $u\times v$-matrix, and if $\left(  i_{1},i_{2},\ldots
,i_{p}\right)  \in\left[  u\right]  ^{p}$ and $\left(  j_{1},j_{2}%
,\ldots,j_{q}\right)  \in\left[  v\right]  ^{q}$ are two sequences of
integers, then $\operatorname*{sub}\nolimits_{i_{1},i_{2},\ldots,i_{p}}%
^{j_{1},j_{2},\ldots,j_{q}}B$ shall denote the $p\times q$-matrix $\left(
B_{i_{x},j_{y}}\right)  _{1\leq x\leq p,\ 1\leq y\leq q}$. If $i_{1}%
<i_{2}<\cdots<i_{p}$ and $j_{1}<j_{2}<\cdots<j_{q}$, then this matrix is a
submatrix of $B$.

\item If $B$ is a $u\times v$-matrix, and if $I$ is a subset of $\left[
u\right]  $, and if $J$ is a subset of $\left[  v\right]  $, then
$\operatorname*{sub}\nolimits_{I}^{J}B$ shall denote the submatrix
$\operatorname*{sub}\nolimits_{i_{1},i_{2},\ldots,i_{p}}^{j_{1},j_{2}%
,\ldots,j_{q}}B$ of $B$, where $i_{1},i_{2},\ldots,i_{p}$ are the elements of
$I$ in increasing order, and where $j_{1},j_{2},\ldots,j_{q}$ are the elements
of $J$ in increasing order.

Thus, in particular, if $B$ is an $n\times n$-matrix, and if $I$ is a subset
of $\left[  n\right]  $, then $\operatorname*{sub}\nolimits_{I}^{I}B$ is a
principal submatrix of $B$, so that $\det\left(  \operatorname*{sub}%
\nolimits_{I}^{I}B\right)  $ is a principal minor of $B$.

\item If $B$ is an $n\times n$-matrix, and if $i,j\in\left[  n\right]  $, then
$B_{\sim i,\sim j}$ shall denote the submatrix of $B$ obtained by removing the
$i$-th row and the $j$-th column from $B$. In other words, $B_{\sim i,\sim j}$
denotes the matrix $\operatorname*{sub}\nolimits_{\left[  n\right]
\setminus\left\{  i\right\}  }^{\left[  n\right]  \setminus\left\{  j\right\}
}B$.

\item If $B$ is an $n\times n$-matrix, then $\operatorname*{adj}B$ shall mean
the \emph{adjugate matrix} of $B$. This is defined as the $n\times n$-matrix
$\left(  \left(  -1\right)  ^{i+j}\det\left(  B_{\sim j,\sim i}\right)
\right)  _{1\leq i\leq n,\ 1\leq j\leq n}$.

\item If $m$ is a nonnegative integer, then $I_{m}$ denotes the $m\times m$
identity matrix.
\end{itemize}

We will need the following properties of determinants:

\begin{itemize}
\item For any $n\times n$-matrix $B$, we have%
\begin{equation}
B\cdot\left(  \operatorname*{adj}B\right)  =\left(  \operatorname*{adj}%
B\right)  \cdot B=\left(  \det B\right)  \cdot I_{n}. \label{eq.adj.BadjB}%
\end{equation}
(This is the main property of adjugates; see, e.g., \cite[Theorem
6.100]{detnotes} for a proof.)

\item For any commutative ring $S$, any $m\times m$-matrix $B$ and any element
$x\in S$, we have%
\begin{equation}
\det\left(  B+xI_{m}\right)  =\sum_{P\subseteq\left[  m\right]  }\det\left(
\operatorname*{sub}\nolimits_{P}^{P}B\right)  \cdot x^{m-\left\vert
P\right\vert }. \label{eq.det(B+xI)0}%
\end{equation}
(This is a folklore result -- essentially the explicit formula for the
characteristic polynomial of a matrix in terms of its principal minors. The
proof is straightforward: Expand the left hand side into a sum of products,
and combine products according to \textquotedblleft which factors come from
$B$ and which factors come from $xI_{m}$\textquotedblright. See
\cite[Proposition 6.4.29]{21s} or \cite[Corollary 6.164]{detnotes} for
detailed proofs.)
\end{itemize}

For any ring $S$, we consider the univariate polynomial ring $S\left[
t\right]  $ as well as the ring $S\left[  \left[  t\right]  \right]  $ of
formal power series. Of course, $S\left[  t\right]  $ is a subring of
$S\left[  \left[  t\right]  \right]  $. Note that the ring $S$ need not be
commutative for $S\left[  t\right]  $ and $S\left[  \left[  t\right]  \right]
$ to be defined.

\section{Proof of Theorem \ref{thm.main2}}

We now finally turn to the proof of Theorem \ref{thm.main2}.

\begin{proof}
[Proof of Theorem \ref{thm.main2}.]We must prove that all diagonal entries of
$A^{m}$ belong to $\mathcal{P}$. In other words, we must prove that $\left(
A^{m}\right)  _{i,i}\in\mathcal{P}$ for each $i\in\left[  n\right]  $.

It is well-known that a polynomial over a matrix ring is \textquotedblleft
essentially the same as\textquotedblright\ a matrix with polynomial entries.
In other words, we can identify the ring $R^{n\times n}\left[  t\right]  $
with the ring $\left(  R\left[  t\right]  \right)  ^{n\times n}$ using a
straightforward ring isomorphism (which sends each $\sum_{i\geq0}C_{i}t^{i}\in
R^{n\times n}\left[  t\right]  $ to $\sum_{i\geq0}C_{i}t^{i}\in\left(
R\left[  t\right]  \right)  ^{n\times n}$). In the same way, we identify the
ring $R^{n\times n}\left[  \left[  t\right]  \right]  $ with the ring $\left(
R\left[  \left[  t\right]  \right]  \right)  ^{n\times n}$.

Let $B$ be the matrix $I_{n}-tA$ in the power series ring $R^{n\times
n}\left[  \left[  t\right]  \right]  $. This matrix $B=I_{n}-tA$ is
invertible, and its inverse is%
\begin{equation}
B^{-1}=I_{n}+tA+t^{2}A^{2}+t^{3}A^{3}+\cdots. \label{eq.In-tA.inv}%
\end{equation}
(This can be proved by directly verifying that $I_{n}+tA+t^{2}A^{2}+t^{3}%
A^{3}+\cdots$ is inverse to $I_{n}-tA$. Indeed, both products $\left(
I_{n}+tA+t^{2}A^{2}+t^{3}A^{3}+\cdots\right)  \cdot\left(  I_{n}-tA\right)  $
and $\left(  I_{n}-tA\right)  \cdot\left(  I_{n}+tA+t^{2}A^{2}+t^{3}%
A^{3}+\cdots\right)  $ turn, upon expanding, into sums that telescope to
$I_{n}$.)

Since the matrix $B$ is invertible, its determinant $\det B$ is invertible as
well (since $\left(  \det B\right)  \cdot\left(  \det\left(  B^{-1}\right)
\right)  =\det\left(  \underbrace{BB^{-1}}_{=I_{n}}\right)  =\det\left(
I_{n}\right)  =1$).

Now, recall that we must prove that $\left(  A^{m}\right)  _{i,i}%
\in\mathcal{P}$ for each $i\in\left[  n\right]  $. So let us fix $i\in\left[
n\right]  $. Then, (\ref{eq.In-tA.inv}) yields%
\begin{align*}
\left(  B^{-1}\right)  _{i,i}  &  =\left(  I_{n}+tA+t^{2}A^{2}+t^{3}%
A^{3}+\cdots\right)  _{i,i}\\
&  =\left(  I_{n}\right)  _{i,i}+tA_{i,i}+t^{2}\left(  A^{2}\right)
_{i,i}+t^{3}\left(  A^{3}\right)  _{i,i}+\cdots.
\end{align*}
Hence, the $t^{m}$-coefficient of the power series $\left(  B^{-1}\right)
_{i,i}\in R\left[  \left[  t\right]  \right]  $ is $\left(  A^{m}\right)
_{i,i}$. Thus, in order to prove that $\left(  A^{m}\right)  _{i,i}%
\in\mathcal{P}$ (which is our goal), it suffices to show that all coefficients
of the power series $\left(  B^{-1}\right)  _{i,i}$ belong to $\mathcal{P}$.
In other words, it suffices to show that $\left(  B^{-1}\right)  _{i,i}%
\in\mathcal{P}\left[  \left[  t\right]  \right]  $. This is what we shall now show.

From (\ref{eq.adj.BadjB}), we obtain $B\cdot\left(  \operatorname*{adj}%
B\right)  =\left(  \operatorname*{adj}B\right)  \cdot B=\left(  \det B\right)
\cdot I_{n}$, so that
\[
B^{-1}=\dfrac{1}{\det B}\cdot\operatorname*{adj}B.
\]
Hence,%
\begin{equation}
\left(  B^{-1}\right)  _{i,i}=\dfrac{1}{\det B}\cdot\left(
\operatorname*{adj}B\right)  _{i,i}. \label{sol.detB-1ii=}%
\end{equation}
Our next goal is to show that both factors $\dfrac{1}{\det B}$ and $\left(
\operatorname*{adj}B\right)  _{i,i}$ on the right hand side of this equality
belong to $\mathcal{P}\left[  \left[  t\right]  \right]  $. This will then
entail that $\left(  B^{-1}\right)  _{i,i}\in\mathcal{P}\left[  \left[
t\right]  \right]  $ as well, and we will be done.

From $B=I_{n}-tA=-tA+1I_{n}$, we obtain
\begin{align}
\det B  &  =\det\left(  -tA+1I_{n}\right)  =\sum_{P\subseteq\left[  n\right]
}\det\left(  \underbrace{\operatorname*{sub}\nolimits_{P}^{P}\left(
-tA\right)  }_{=-t\operatorname*{sub}\nolimits_{P}^{P}A}\right)
\cdot\underbrace{1^{n-\left\vert P\right\vert }}_{=1}\nonumber\\
&  \ \ \ \ \ \ \ \ \ \ \ \ \ \ \ \ \ \ \ \ \left(
\begin{array}
[c]{c}%
\text{by (\ref{eq.det(B+xI)0}), applied to }n\text{, }R\left[  \left[
t\right]  \right]  \text{, }-tA\text{ and }1\\
\text{instead of }m\text{, }S\text{, }B\text{ and }x
\end{array}
\right) \nonumber\\
&  =\sum_{P\subseteq\left[  n\right]  }\underbrace{\det\left(
-t\operatorname*{sub}\nolimits_{P}^{P}A\right)  }_{=\left(  -t\right)
^{\left\vert P\right\vert }\det\left(  \operatorname*{sub}\nolimits_{P}%
^{P}A\right)  }\nonumber\\
&  =\sum_{P\subseteq\left[  n\right]  }\left(  -t\right)  ^{\left\vert
P\right\vert }\underbrace{\det\left(  \operatorname*{sub}\nolimits_{P}%
^{P}A\right)  }_{\substack{\in\mathcal{P}\\\text{(since }\det\left(
\operatorname*{sub}\nolimits_{P}^{P}A\right)  \text{ is}\\\text{a principal
minor of }A\text{)}}}\label{sol.detB=.1}\\
&  \in\mathcal{P}\left[  t\right]  \subseteq\mathcal{P}\left[  \left[
t\right]  \right]  .\nonumber
\end{align}
Thus, $\det B$ is a formal power series over $\mathcal{P}$. Moreover,
(\ref{sol.detB=.1}) shows that this power series has constant term $1$ (since
the only addend in the sum in (\ref{sol.detB=.1}) that contributes to the
constant term is the addend for $P=\varnothing$, but this addend is
$\underbrace{\left(  -t\right)  ^{\left\vert \varnothing\right\vert }}%
_{=1}\underbrace{\det\left(  \operatorname*{sub}\nolimits_{\varnothing
}^{\varnothing}A\right)  }_{\substack{=1\\\text{(since the }0\times
0\text{-matrix}\\\text{has determinant }1\text{)}}}=1$). Thus, this power
series is invertible in $\mathcal{P}\left[  \left[  t\right]  \right]  $.
Therefore,%
\begin{equation}
\dfrac{1}{\det B}\in\mathcal{P}\left[  \left[  t\right]  \right]  .
\label{sol.1/detB-in}%
\end{equation}

Now, recall the definition of an adjugate matrix. This definition yields%
\begin{align}
\left(  \operatorname*{adj}B\right)  _{i,i}  &  =\underbrace{\left(
-1\right)  ^{i+i}}_{=1}\det\left(  B_{\sim i,\sim i}\right)  =\det\left(
B_{\sim i,\sim i}\right) \nonumber\\
&  =\det\left(  \left(  -tA+1I_{n}\right)  _{\sim i,\sim i}\right)
\ \ \ \ \ \ \ \ \ \ \left(  \text{since }B=-tA+1I_{n}\right) \nonumber\\
&  =\det\left(  -tA_{\sim i,\sim i}+1I_{n-1}\right)
\ \ \ \ \ \ \ \ \ \ \left(  \text{since }\left(  -tA+1I_{n}\right)  _{\sim
i,\sim i}=-tA_{\sim i,\sim i}+1I_{n-1}\right) \nonumber\\
&  =\sum_{P\subseteq\left[  n-1\right]  }\det\left(
\underbrace{\operatorname*{sub}\nolimits_{P}^{P}\left(  -tA_{\sim i,\sim
i}\right)  }_{=-t\operatorname*{sub}\nolimits_{P}^{P}\left(  A_{\sim i,\sim
i}\right)  }\right)  \cdot\underbrace{1^{n-1-\left\vert P\right\vert }}%
_{=1}\nonumber\\
&  \ \ \ \ \ \ \ \ \ \ \ \ \ \ \ \ \ \ \ \ \left(
\begin{array}
[c]{c}%
\text{by (\ref{eq.det(B+xI)0}), applied to }n-1\text{, }R\left[  \left[
t\right]  \right]  \text{, }-tA_{\sim i,\sim i}\text{ and }1\\
\text{instead of }m\text{, }S\text{, }B\text{ and }x
\end{array}
\right) \nonumber\\
&  =\sum_{P\subseteq\left[  n-1\right]  }\underbrace{\det\left(
-t\operatorname*{sub}\nolimits_{P}^{P}\left(  A_{\sim i,\sim i}\right)
\right)  }_{=\left(  -t\right)  ^{\left\vert P\right\vert }\det\left(
\operatorname*{sub}\nolimits_{P}^{P}\left(  A_{\sim i,\sim i}\right)  \right)
}\nonumber\\
&  =\sum_{P\subseteq\left[  n-1\right]  }\left(  -t\right)  ^{\left\vert
P\right\vert }\det\left(  \operatorname*{sub}\nolimits_{P}^{P}\left(  A_{\sim
i,\sim i}\right)  \right)  . \label{sol.adjBii=.1}%
\end{align}

Now, let $P$ be an arbitrary subset of $\left[  n-1\right]  $. Write this
subset $P$ in the form $P=\left\{  p_{1},p_{2},\ldots,p_{r}\right\}  $, where
$p_{1}<p_{2}<\cdots<p_{r}$. Furthermore, let $g\in\left\{  0,1,\ldots
,r\right\}  $ be the element that satisfies%
\[
p_{1}<p_{2}<\cdots<p_{g}<i\leq p_{g+1}<p_{g+2}<\cdots<p_{r}.
\]
(Here, $g$ will be $0$ if all elements of $P$ are $\geq i$, and $g$ will be
$r$ if all elements of $P$ are $<i$.) Then, due to the combinatorial nature of
removing rows and columns, we have%
\[
\operatorname*{sub}\nolimits_{P}^{P}\left(  A_{\sim i,\sim i}\right)
=\operatorname*{sub}\nolimits_{P^{\prime}}^{P^{\prime}}A,
\]
where $P^{\prime}$ is the subset $\left\{  p_{1},p_{2},\ldots,p_{g}\right\}
\cup\left\{  p_{g+1}+1,p_{g+2}+1,\ldots,p_{r}+1\right\}  $ of $\left[
n\right]  $. Hence, $\operatorname*{sub}\nolimits_{P}^{P}\left(  A_{\sim
i,\sim i}\right)  $ is a principal submatrix of $A$. Therefore, its
determinant $\det\left(  \operatorname*{sub}\nolimits_{P}^{P}\left(  A_{\sim
i,\sim i}\right)  \right)  $ is a principal minor of $A$, thus belongs to
$\mathcal{P}$.

Forget that we fixed $P$. We thus have shown that $\det\left(
\operatorname*{sub}\nolimits_{P}^{P}\left(  A_{\sim i,\sim i}\right)  \right)
\in\mathcal{P}$ for each $P\subseteq\left[  n-1\right]  $. Therefore,
(\ref{sol.adjBii=.1}) becomes%
\begin{equation}
\left(  \operatorname*{adj}B\right)  _{i,i}=\sum_{P\subseteq\left[
n-1\right]  }\left(  -t\right)  ^{\left\vert P\right\vert }\underbrace{\det
\left(  \operatorname*{sub}\nolimits_{P}^{P}\left(  A_{\sim i,\sim i}\right)
\right)  }_{\in\mathcal{P}}\in\mathcal{P}\left[  t\right]  \subseteq
\mathcal{P}\left[  \left[  t\right]  \right]  . \label{sol.adjBii-in}%
\end{equation}

Now, (\ref{sol.detB-1ii=}) becomes%
\[
\left(  B^{-1}\right)  _{i,i}=\underbrace{\dfrac{1}{\det B}}_{\substack{\in
\mathcal{P}\left[  \left[  t\right]  \right]  \\\text{(by (\ref{sol.1/detB-in}%
))}}}\cdot\underbrace{\left(  \operatorname*{adj}B\right)  _{i,i}%
}_{\substack{\in\mathcal{P}\left[  \left[  t\right]  \right]  \\\text{(by
(\ref{sol.adjBii-in}))}}}\in\mathcal{P}\left[  \left[  t\right]  \right]
\cdot\mathcal{P}\left[  \left[  t\right]  \right]  \subseteq\mathcal{P}\left[
\left[  t\right]  \right]  .
\]
As explained above, this completes our proof of Theorem \ref{thm.main2}.
\end{proof}

Somewhat regrettably, the above proof is the slickest I am aware of. A
more-or-less equivalent proof can be given avoiding the use of power series
(using \cite[Proposition 3.9 and Lemma 3.11]{trach} instead). A more
pedestrian (but harder to formalize) proof uses the Cayley--Hamilton theorem
and a variant of the inclusion/exclusion principle.

\section{Variants}

A counterpart of Theorem \ref{thm.main2} for the off-diagonal entries of
$A^{m}$ exists as well:

\begin{theorem}
\label{thm.main2off}Let $n$, $m$, $R$, $A$ and $\mathcal{P}$ be as in Theorem
\ref{thm.main2}.

Let $i$ and $j$ be two distinct elements of $\left[  n\right]  $. An $\left(
i,j\right)  $\emph{-quasiprincipal minor} of $A$ shall mean a determinant of
the form $\det\left(  \operatorname*{sub}\nolimits_{I}^{J}A\right)  $, where
$I$ and $J$ are two subsets of $\left[  n\right]  $ satisfying%
\[
i\in I\text{ and }j\in J\text{ and }\left\vert I\right\vert =\left\vert
J\right\vert \text{ and }J=\left(  I\setminus\left\{  i\right\}  \right)
\cup\left\{  j\right\}  .
\]
(For instance, if $n\geq7$, then $\det\left(  \operatorname*{sub}%
\nolimits_{\left\{  1,2,7\right\}  }^{\left\{  2,5,7\right\}  }A\right)  $ is
a $\left(  1,5\right)  $-quasiprincipal minor of $A$.)

Let $\mathcal{K}_{i,j}$ be the $\mathbb{Z}$-submodule of $R$ spanned by all
$\left(  i,j\right)  $-quasiprincipal minors of $A$. Then,
\[
\left(  A^{m}\right)  _{i,j}\in\mathcal{P}\cdot\mathcal{K}_{i,j}.
\]
(Here, $\mathcal{P}\cdot\mathcal{K}_{i,j}$ denotes the $\mathbb{Z}$-linear
span of all products of the form $pk$ with $p\in\mathcal{P}$ and
$k\in\mathcal{K}_{i,j}$.)
\end{theorem}

\begin{proof}
[Proof outline.]This is similar to our above proof of Theorem \ref{thm.main2},
but some changes are needed. Most importantly, instead of proving that
$\left(  \operatorname*{adj}B\right)  _{i,i}\in\mathcal{P}\left[  \left[
t\right]  \right]  $, we now need to show that $\left(  \operatorname*{adj}%
B\right)  _{i,j}\in\mathcal{K}_{i,j}\left[  \left[  t\right]  \right]  $ (that
is, that all coefficients of the power series $\left(  \operatorname*{adj}%
B\right)  _{i,j}$ belong to $\mathcal{K}_{i,j}$). To do so, we apply the
definition of the adjugate matrix to see that
\begin{equation}
\left(  \operatorname*{adj}B\right)  _{i,j}=\left(  -1\right)  ^{j+i}%
\det\left(  B_{\sim j,\sim i}\right)  . \label{pf.thm.main2off.1}%
\end{equation}
We can simplify $B_{\sim j,\sim i}$ further to $-tA_{\sim j,\sim i}+\left(
I_{n}\right)  _{\sim j,\sim i}$ (since $B=I_{n}-tA=-tA+I_{n}$), but
unfortunately this is not the same as $-tA_{\sim j,\sim i}+1I_{n-1}$, and thus
we can no longer apply (\ref{eq.det(B+xI)0}). Instead, we use a trick:

\begin{itemize}
\item We define $A^{\prime}$ to be the matrix obtained from $-tA$ by replacing
the $j$-th row by $\left(  0,0,\ldots,0,1,0,0,\ldots,0\right)  $, where the
only entry equal to $1$ is in the $i$-th position.

\item We define $I_{n}^{\prime}$ to be the matrix obtained from $I_{n}$ by
replacing the $1$ in the $j$-th row by a $0$.

\item We define $B^{\prime}$ to be the matrix obtained from $B$ by replacing
the $j$-th row by $\left(  0,0,\ldots,0,1,0,0,\ldots,0\right)  $, where the
only entry equal to $1$ is in the $i$-th position.
\end{itemize}

Laplace expansion along the $j$-th row shows that
\[
\det\left(  B^{\prime}\right)  =\left(  -1\right)  ^{j+i}\det\left(  \left(
B^{\prime}\right)  _{\sim j,\sim i}\right)  =\left(  -1\right)  ^{j+i}%
\det\left(  B_{\sim j,\sim i}\right)
\]
(since the matrix $B^{\prime}$ differs from $B$ only in the $j$-th row, and
thus we have $\left(  B^{\prime}\right)  _{\sim j,\sim i}=B_{\sim j,\sim i}$).
Comparing this with (\ref{pf.thm.main2off.1}), we find%
\begin{equation}
\left(  \operatorname*{adj}B\right)  _{i,j}=\det\left(  B^{\prime}\right)  .
\label{pf.thm.main2off.2}%
\end{equation}

Furthermore, recall that $B=I_{n}-tA=-tA+I_{n}$. Thus, $B^{\prime}=A^{\prime
}+I_{n}^{\prime}$ (based on how $A^{\prime}$, $I_{n}^{\prime}$ and $B^{\prime
}$ were constructed).

On the other hand, the definition of $I_{n}^{\prime}$ shows that
$I_{n}^{\prime}$ is a diagonal $n\times n$-matrix with diagonal entries
$1,1,\ldots,1,0,1,1,\ldots,1$, where the only diagonal entry equal to $0$ is
in the $j$-th position. However, another classical fact about determinants
(\cite[Theorem 6.4.26]{21s}, \cite[Corollary 6.162]{detnotes}) shows that if
$C$ is any $n\times n$-matrix, and if $D$ is a diagonal $n\times n$-matrix
with diagonal entries $d_{1},d_{2},\ldots,d_{n}$, then%
\[
\det\left(  C+D\right)  =\sum_{P\subseteq\left[  n\right]  }\det\left(
\operatorname*{sub}\nolimits_{P}^{P}C\right)  \cdot\prod_{k\in\left[
n\right]  \setminus P}d_{k}.
\]
We can apply this to $C=A^{\prime}$ and $D=I_{n}^{\prime}$ and $\left(
d_{1},d_{2},\ldots,d_{n}\right)  =\underbrace{\left(  1,1,\ldots
,1,0,1,1,\ldots,1\right)  }_{\text{the }0\text{ is in the }j\text{-th
position}}$, and thus obtain%
\begin{align*}
&  \det\left(  A^{\prime}+I_{n}^{\prime}\right) \\
&  =\sum_{P\subseteq\left[  n\right]  }\det\left(  \operatorname*{sub}%
\nolimits_{P}^{P}\left(  A^{\prime}\right)  \right)  \cdot\prod_{k\in\left[
n\right]  \setminus P}%
\begin{cases}
1, & \text{if }k\neq j;\\
0, & \text{if }k=j
\end{cases}
\\
&  \ \ \ \ \ \ \ \ \ \ \ \ \ \ \ \ \ \ \ \ \left(  \text{since the }k\text{-th
diagonal entry of }I_{n}^{\prime}\text{ is }%
\begin{cases}
1, & \text{if }k\neq j;\\
0, & \text{if }k=j
\end{cases}
\right) \\
&  =\sum_{P\subseteq\left[  n\right]  }\det\left(  \operatorname*{sub}%
\nolimits_{P}^{P}\left(  A^{\prime}\right)  \right)  \cdot%
\begin{cases}
1, & \text{if }j\notin\left[  n\right]  \setminus P;\\
0, & \text{if }j\in\left[  n\right]  \setminus P
\end{cases}
\\
&  =\sum_{\substack{P\subseteq\left[  n\right]  ;\\j\notin\left[  n\right]
\setminus P}}\det\left(  \operatorname*{sub}\nolimits_{P}^{P}\left(
A^{\prime}\right)  \right)  =\sum_{\substack{P\subseteq\left[  n\right]
;\\j\in P}}\det\left(  \operatorname*{sub}\nolimits_{P}^{P}\left(  A^{\prime
}\right)  \right) \\
&  =\sum_{\substack{P\subseteq\left[  n\right]  ;\\j\in P\text{ and }i\in
P}}\det\left(  \operatorname*{sub}\nolimits_{P}^{P}\left(  A^{\prime}\right)
\right)  +\sum_{\substack{P\subseteq\left[  n\right]  ;\\j\in P\text{ and
}i\notin P}}\underbrace{\det\left(  \operatorname*{sub}\nolimits_{P}%
^{P}\left(  A^{\prime}\right)  \right)  }_{\substack{=0\\\text{(since all
entries in the }j\text{-th row of }A^{\prime}\\\text{are }0\text{ except for
the }i\text{-th entry, and thus}\\\text{the matrix }\operatorname*{sub}%
\nolimits_{P}^{P}\left(  A^{\prime}\right)  \text{ has a zero row)}}}\\
&  =\sum_{\substack{P\subseteq\left[  n\right]  ;\\j\in P\text{ and }i\in
P}}\underbrace{\det\left(  \operatorname*{sub}\nolimits_{P}^{P}\left(
A^{\prime}\right)  \right)  }_{\substack{=\pm\det\left(  \operatorname*{sub}%
\nolimits_{P\setminus\left\{  j\right\}  }^{P\setminus\left\{  i\right\}
}\left(  -tA\right)  \right)  \\\text{(by Laplace expansion along the
row}\\\text{that was the }j\text{-th row of }A^{\prime}\text{)}}%
}=\sum_{\substack{P\subseteq\left[  n\right]  ;\\j\in P\text{ and }i\in P}%
}\pm\underbrace{\det\left(  \operatorname*{sub}\nolimits_{P\setminus\left\{
j\right\}  }^{P\setminus\left\{  i\right\}  }\left(  -tA\right)  \right)
}_{=\pm t^{\left\vert P\right\vert -1}\det\left(  \operatorname*{sub}%
\nolimits_{P\setminus\left\{  j\right\}  }^{P\setminus\left\{  i\right\}
}A\right)  }\\
&  =\sum_{\substack{P\subseteq\left[  n\right]  ;\\j\in P\text{ and }i\in
P}}\pm t^{\left\vert P\right\vert -1}\underbrace{\det\left(
\operatorname*{sub}\nolimits_{P\setminus\left\{  j\right\}  }^{P\setminus
\left\{  i\right\}  }A\right)  }_{\substack{\in\mathcal{K}_{i,j}\\\text{(since
}\det\left(  \operatorname*{sub}\nolimits_{P\setminus\left\{  j\right\}
}^{P\setminus\left\{  i\right\}  }A\right)  \text{ is}\\\text{an }\left(
i,j\right)  \text{-quasiprincipal minor of }A\text{)}}}\in\sum
_{\substack{P\subseteq\left[  n\right]  ;\\j\in P\text{ and }i\in P}}\pm
t^{\left\vert P\right\vert -1}\mathcal{K}_{i,j}\subseteq\mathcal{K}%
_{i,j}\left[  \left[  t\right]  \right]  .
\end{align*}
In view of $B^{\prime}=A^{\prime}+I_{n}^{\prime}$, this rewrites as
$\det\left(  B^{\prime}\right)  \in\mathcal{K}_{i,j}\left[  \left[  t\right]
\right]  $. Hence, (\ref{pf.thm.main2off.2}) becomes $\left(
\operatorname*{adj}B\right)  _{i,j}=\det\left(  B^{\prime}\right)
\in\mathcal{K}_{i,j}\left[  \left[  t\right]  \right]  $. Having shown this,
we can finish the proof as we did for Theorem \ref{thm.main2}.
\end{proof}

Another variant of Theorem \ref{thm.main2} is the following:

\begin{theorem}
\label{thm.main3}Let $n$ and $m$ be nonnegative integers. Let $R$ be a
commutative ring. Let $A$ be an $n\times n$-matrix over $R$. Assume that all
principal minors of $A$ equal $1$. Then, all diagonal entries of $A^{m}$ equal
$1$.
\end{theorem}

\begin{proof}
Follow the above proof of Theorem \ref{thm.main2}. From (\ref{sol.detB=.1}),
we obtain%
\[
\det B=\sum_{P\subseteq\left[  n\right]  }\left(  -t\right)  ^{\left\vert
P\right\vert }\underbrace{\det\left(  \operatorname*{sub}\nolimits_{P}%
^{P}A\right)  }_{\substack{=1\\\text{(by assumption, since }\det\left(
\operatorname*{sub}\nolimits_{P}^{P}A\right)  \\\text{is a principal minor of
}A\text{)}}}=\sum_{P\subseteq\left[  n\right]  }\left(  -t\right)
^{\left\vert P\right\vert }=\left(  1-t\right)  ^{n}%
\]
(since the binomial formula yields $\left(  1-t\right)  ^{n}=\sum_{k=0}%
^{n}\dbinom{n}{k}\left(  -t\right)  ^{k}=\sum_{P\subseteq\left[  n\right]
}\left(  -t\right)  ^{\left\vert P\right\vert }$). Let $i\in\left\{
1,2,\ldots,n\right\}  $. From (\ref{sol.adjBii=.1}), we obtain%
\[
\left(  \operatorname*{adj}B\right)  _{i,i}=\sum_{P\subseteq\left[
n-1\right]  }\left(  -t\right)  ^{\left\vert P\right\vert }\underbrace{\det
\left(  \operatorname*{sub}\nolimits_{P}^{P}\left(  A_{\sim i,\sim i}\right)
\right)  }_{\substack{=1\\\text{(by assumption,}\\\text{since }\det\left(
\operatorname*{sub}\nolimits_{P}^{P}\left(  A_{\sim i,\sim i}\right)  \right)
\\\text{is a principal minor of }A\text{)}}}=\sum_{P\subseteq\left[
n-1\right]  }\left(  -t\right)  ^{\left\vert P\right\vert }=\left(
1-t\right)  ^{n-1}%
\]
(again by the binomial formula). Now, (\ref{sol.detB-1ii=}) becomes%
\begin{align*}
\left(  B^{-1}\right)  _{i,i}  &  =\dfrac{1}{\det B}\cdot\left(
\operatorname*{adj}B\right)  _{i,i}=\underbrace{\left(  \operatorname*{adj}%
B\right)  _{i,i}}_{=\left(  1-t\right)  ^{n-1}}\diagup\underbrace{\left(  \det
B\right)  }_{=\left(  1-t\right)  ^{n}}=\left(  1-t\right)  ^{n-1}%
\diagup\left(  1-t\right)  ^{n}\\
&  =\dfrac{1}{1-t}=1+t+t^{2}+t^{3}+\cdots.
\end{align*}
Thus, the $t^{m}$-coefficient of the power series $\left(  B^{-1}\right)
_{i,i}\in R\left[  \left[  t\right]  \right]  $ is $1$. However, we have
already seen that this coefficient is $\left(  A^{m}\right)  _{i,i}$. Thus, we
conclude that $\left(  A^{m}\right)  _{i,i}=1$. This shows that all diagonal
entries of $A^{m}$ equal $1$, so that Theorem \ref{thm.main3} is proved.
\end{proof}

\section{Back to Putnam 2021}

As already mentioned, we do not know whether Theorem \ref{thm.putnam} can be
generalized by replacing \textquotedblleft odd\textquotedblright\ by
\textquotedblleft congruent to $1$ modulo $4$\textquotedblright. More
generally, we are tempted to ask the following:

\begin{question}
\label{quest.putnam-R}Fix a commutative ring $R$. Let $A$ be an $n\times
n$-matrix over $R$. Let $m$ be a nonnegative integer. Assume that each
principal minor of $A$ is $1$. Is it true that each principal minor of $A^{m}$
is $1$ as well?
\end{question}

For $R=\mathbb{Z}/2$, this would yield Theorem \ref{thm.putnam}; the
\textquotedblleft congruent to $1$ modulo $4$\textquotedblright\ variant would
follow for $R=\mathbb{Z}/4$. Theorem \ref{thm.main3} corresponds to the case
when the principal minor of $A^{m}$ is a diagonal entry. The argument from
\cite[Lemma 3.3]{MasPan} shows that Question \ref{quest.putnam-R} has a
positive answer whenever $R$ is an integral domain; thus, the answer is also
positive when $R$ is a product of integral domains. On the other hand, if $R$
can be arbitrary, then the answer to Question \ref{quest.putnam-R} is
negative, but the only counterexample we know is when $R$ is a certain
quotient ring of a polynomial ring\footnote{Here are the details: Let $R$ be
the quotient ring%
\[
\mathbb{Q}\left[  x,y\right]  \diagup\left(  x^{3}+y^{3},xy,x^{4},x^{3}%
y,x^{2}y^{2},xy^{3},y^{4}\right)  ,
\]
and let $A:=\left(
\begin{array}
[c]{cccc}%
1 & 1 & 0 & 0\\
0 & 1 & y & x\\
x & 0 & 1 & y\\
y & 0 & x & 1
\end{array}
\right)  \in R^{4\times4}$. Then, all principal minors of $A$ are $1$, but the
principal minor $\det\left(  \operatorname*{sub}\nolimits_{\left\{
2,3\right\}  }^{\left\{  2,3\right\}  }\left(  A^{2}\right)  \right)
=1-x^{3}-xy$ is not $1$ since $x^{3}\neq0$ in $R$. Actually, we can replace
$\mathbb{Q}$ by any field here (even by $\mathbb{Z}/2$); if this field is
finite, then $R$ becomes a finite ring. (But we cannot turn $R$ into
$\mathbb{Z}/n$ without changing the construction of $A$.)} (and $n=4$ and
$m=2$). The smallest ring $R$ for which the question remains open is
$\mathbb{Z}/4$.

\begin{statement}
\textbf{Update (2026):} Question \ref{quest.putnam-R} has a positive answer
for all rings $R$ of the form $\mathbb{Z}/d$ with $d\in\mathbb{Z}$, and for
many others as well. A proof -- found by the GPT-5.5 LLM -- is available at
\cite{princmins2}.
\end{statement}


\begin{thebibliography}{9}                                                                                                %


\bibitem {Kedlaya}Kiran Kedlaya, \textit{The Putnam Archive},
\url{https://kskedlaya.org/putnam-archive/} .

\bibitem {MasPan}Mikiya Masuda, Taras Panov, \textit{Semifree circle actions,
Bott towers, and quasitoric manifolds}, Sbornik Math. \textbf{199} (2008), no.
7-8, 1201-1223,
\href{https://arxiv.org/abs/math/0607094v2}{arXiv:math/0607094v2}.

\bibitem {Putnam-official-2021}MAA, \textit{Putnam Competition}, official
website, \url{https://www.maa.org/math-competitions/putnam-competition} .

\bibitem {HarLoe}\href{https://doi.org/10.1016/0024-3795(84)90209-x}{D. J.
Hartfiel and R. Loewy, \textit{On matrices having equal corresponding
principal minors}, Linear Algebra and its Applications \textbf{56} (1984), pp.
147--167}.

\bibitem {detnotes}\href{https://arxiv.org/abs/2008.09862v3}{Darij Grinberg,
\textit{Notes on the combinatorial fundamentals of algebra}, 15 September
2022, arXiv:2008.09862v3}.

\bibitem {trach}\href{https://arxiv.org/abs/2510.20689v1}{Darij Grinberg,
\textit{The trace Cayley-Hamilton theorem}, arXiv:2510.20689v1}.\newline\url{https://www.cip.ifi.lmu.de/~grinberg/algebra/trach.pdf}

\bibitem {21s}Darij Grinberg, \textit{An Introduction to Algebraic
Combinatorics [Math 701, Spring 2021 lecture notes]}, 30 January 2022.
\newline\url{https://www.cip.ifi.lmu.de/~grinberg/t/21s/lecs.pdf}

\bibitem {princmins2}\href{https://arxiv.org/abs/2606.28976v1}{Darij Grinberg,
Hamesh M. Hamesh, \textit{Powers of
matrices with all principal minors equal to }$1$,
arXiv:2606.28976v1.}
\newline
Also available at
\url{https://www.cip.ifi.lmu.de/~grinberg/algebra/princmins2gpt.pdf}
\end{thebibliography}
\end{document}